\documentclass[12pt]{amsart}
\usepackage{amssymb}
\usepackage[utf8]{inputenc}
\usepackage{enumerate}
\usepackage{amsmath}
\usepackage{graphicx}

\makeatletter
\@namedef{subjclassname@2010}{%
  \textup{2010} Mathematics Subject Classification}
\makeatother
\newtheorem{thm}{Theorem}

\newtheorem{prop}[thm]{Proposition}
\newtheorem{lem}[thm]{Lemma}
\theoremstyle{definition}
\newtheorem{defin}[thm]{Definition}

\newtheorem{exa}[thm]{Example}

\newcommand{\lns}[1]{\left(\ell_{#1}\left(n!\right)\right)_{n\in \mathbb{N}}}

\begin{document}

\baselineskip=17pt

\title{Automaticity of the sequence of the last nonzero digits of $n!$ in a fixed base}

\author{Eryk Lipka}
\address{Graduate of Institute of Mathematics\\ Jagiellonian University\\
Kraków, Poland}
\email{eryklipka0@gmail.com}

\begin{abstract}
In 2011 Deshouillers and Ruzsa (\cite{deshruzsa}) tried to argument that the sequence of the last nonzero digit of $n!$ in base 12 is not automatic. This statement was proved few years later by Deshoulliers in \cite{desh2}. In this paper we provide alternate proof that lets us generalize the problem and give an exact characterization in which bases the sequence of the last nonzero digits of $n!$ is automatic.
\end{abstract}


\maketitle
\section{Introduction}
Let $\lns{b}$ be the sequence of last nonzero digits of $n!$ in base $b$, in this paper we will answer the question for which values of $b$ is this sequence automatic. It was known that $\lns{b}$ is automatic in many cases including bases being primes or powers of primes, one can also prove that $\lns{b}$ is automatic for some small bases that have more prime factors, like 6 or 10. In general, for base of the form $b=p_1^{a_1}p_2^{a_2}$ where $p_1\neq p_2,\; p_1, p_2 \in \mathbb{P},\; a_1,a_2\in\mathbb{N}_+$, it can be shown that $\lns{b}$ is automatic when $a_1\left(p_1-1\right)\neq a_2\left(p_2-1\right)$. The smallest base for which the answer is unclear is 12. This was the case analysed by Deshouillers and Ruzsa in \cite{deshruzsa}. They conjectured that $\lns{12}$ can not be automatic, despite the fact that it is equal to some automatic sequence nearly everywhere. An attempt to prove that conjecture was done by Deshouliers in his paper \cite{desh1}, and few years later he answered the question by proving the following, stronger result
\begin{thm}{(Deshouillers \cite{desh2})}
For $a \in \{3, 6, 9\}$, the characteristic sequence of $\{n\, ; \ell_{12}\left(n!\right) = a\}$ is not automatic.
\end{thm}
Another way of proving similar fact (but for $a \in \{4,8\}$) was provided recently by Byszewski and Konieczny in \cite{kuba}. It seems that both proofs can be generalized to all cases when $a_1\left(p_1-1\right)=a_2\left(p_2-1\right)$, however it is not obvious if, or how, can it be extended to bases with more than two prime factors. This was our main motivation to write this paper, and we provide complete characterization in which bases is this sequence automatic, including those with many prime factors.

In this paper we will use the following notation: the string of digits of $n$ in base $k$ will be denoted $\left[n\right]_k$, by $v_b\left(n\right)$ we mean the largest integer $t$ such that $b^t|n$, and $s_b\left(n\right)$ is the sum of digits of $n$ in base $b$.
This paper is composed of two main parts, first we recall some basic facts about automatic sequences for readers not familiar with the topic, in the latter part we present our results about automaticity of $\lns{b}$.

We would like to thank Piotr Miska and Maciej Ulas for proof reading and helpful suggestions while preparing this paper.

\section{Basics of automatic sequences}
In this section we will give short summary of topics from automatic sequence theory that we will be using later. If the reader is interested in getting more insight into this topic, we strongly recommend book of Allouche and Shallit \cite{alloucheshallit} that covers all important topics in this area.
\begin{defin}
{\bf Deterministic finite automaton with output} is a 6-tuple $\left(Q, \Sigma, \rho, q_0, \Delta, \tau \right)$ such that
\begin{itemize}
\item$Q$ is a {\bf finite} set of states;
\item$\Sigma$ is an input alphabet;
\item$\rho: Q\times\Sigma \rightarrow Q$ is a transition function;
\item$q_0 \in Q$ is an initial state;
\item$\Delta$ is an output alphabet (finite set);
\item$\tau:Q\rightarrow\Delta$ is an output function.
\end{itemize}
Transition function can be generalized to take strings of characters instead of single ones. For string $s_1s_2s_3\ldots$ we define $\rho\left(q, s_1s_2s_3\ldots\right)=\rho\left(\ldots\rho\left(\rho\left(q, s_1\right), s_2\right)\ldots\right)$.
\end{defin}

\begin{defin}
For any finite alphabet $\Sigma$, function  $f: \Sigma^* \rightarrow \Delta$ is called a {\bf finite-state function} if there exists a deterministic finite automaton with output $\left(Q, \Sigma, \rho, q_0, \Delta, \tau \right)$ such that $f(\omega) = \tau\left(\rho\left(q_0,\omega\right)\right)$.
\end{defin}

\begin{lem} \label{reverse}
If $f: \Sigma^* \rightarrow \Delta$ is a finite-state function then function $g: \Sigma^* \rightarrow \Delta$ defined as $g\left(\omega\right) = f\left(\omega^R\right)$ is also finite-state. ($ ^R$ denotes taking reverse of a word).
\end{lem}
\begin{proof}
(Sketch) Let $\left(Q, \Sigma, \rho, q_0, \Delta, \tau \right)$ be automaton that is related to $f$, we will define another automaton $\left(Q', \Sigma, \rho', q_0', \Delta, \tau' \right)$. Let $Q' = \Delta^Q$ be all functions from $Q$ to $\Delta$ and $q_0'\equiv \tau$. For any $g\in Q'$ we define $\tau'\left(g\right)=g\left(q_0\right)$, and for any $\sigma \in \Sigma, q\in Q$ we put $\rho'\left(g,\sigma\right)\left(q\right)=g\left(\rho\left(q,\sigma\right)\right)$. By induction on length of word $\omega\in\Sigma^*$ one can prove that equation 
$$\rho'\left(g,\omega\right)\left(q\right)=g\left(\rho\left(q,\omega^R\right)\right)$$
holds for any $g \in Q', q\in Q$. And finally $$g\left(\omega\right)=\tau'\left(\rho'\left(q_0',\omega\right)\right)=\rho'\left(q_0',\omega\right)\left(q_0\right)=g\left(\rho\left(q_0,\omega^R\right)\right)=f\left(\omega^R\right).$$
\end{proof}

\begin{defin}
$\left(a\left(n\right)\right)_{n\in\mathbb{N}}$ is an {\bf $k$-automatic sequence} if function $\left[n\right]_k \rightarrow a_n$ is finite-state. By Lemma \ref{reverse} it is not important whether we read representation of $n$ from the right or from the left side.
\end{defin}

Now we present some simple examples of sequences that are automatic.
\begin{exa} \label{simpleexa}
The sequence $a_n = n \pmod{m}$ is $k$-automatic for any $k\geq 2, m\in \mathbb{Z}_+$. In order to see this it is enough to take $Q=\left\{0,1,\ldots,m-1\right\}, \rho\left(q,\sigma\right) = kq+\sigma \pmod{m}$ and read input "from left to right".\\
The sequence $a_n = s_k\left(n\right) \pmod{m}$ is $k$-automatic for any $k\geq 2, m\in \mathbb{Z}_+$. Take $Q=\left\{0,1,\ldots,m-1\right\}$ and $\rho\left(q,\sigma\right) = q+\sigma \pmod{m}$.\\
For any $k\geq 2$ and $x\in \mathbb{N}$, the characteristic sequence $a_n = \delta_x\left(n\right)$ is $k$-automatic. Automaton that computes it can be constructed by taking $\lceil \log_k\left(x\right) \rceil$ states that count how many digits were correct plus one "sinkhole" state that accepts all numbers other than $x$.
\end{exa}
We can also obtain automatic sequences by modifying existing ones.
\begin{exa} 
If $\left(a\left(n\right)\right)_{n\in\mathbb{N}}$ is $k$-automatic sequence then so is $b_n = f(a_n)$ for any function $f$ taking values from the image of $a_n$. The difference will be only in the output function of related automaton.\\
If $\left(a\left(n\right)\right)_{n\in\mathbb{N}}, \left(b\left(n\right)\right)_{n\in\mathbb{N}}$ are $k$-automatic sequences, then so is $c_n = f(a_n, b_n)$ for any function $f$ as long as it is well defined on all possible pairs $(a_n, b_n)$. To obtain such automaton $\left(Q_c, \Sigma, \rho_c, q_c, \Delta_c, \tau_c \right)$ we can take the "product" of automatons $\left(Q_a, \Sigma, \rho_a, q_a, \Delta_a, \tau_a \right)$, $\left(Q_b, \Sigma, \rho_b, q_b, \Delta_b, \tau_b \right)$ defined by
\begin{itemize}
\item $Q_c=Q_a \times Q_b$;
\item $\rho_c\left(a,b\right) = \left(\rho_a\left(a\right),\rho_b\left(b\right)\right)$;
\item $q_c = \left(q_a,q_b\right)$;
\item $\Delta_c=f(\Delta_a\times \Delta_b)$;
\item $\tau_c\left(a,b\right) = f\left(\tau_a\left(a\right),\tau_b\left(b\right)\right)$.
\end{itemize}
This can be easily generalized to the case with $f$ taking any finite number of sequences as an input.
\end{exa}

By combining above examples together we can get some additional facts.
\begin{lem} Let $k\in \mathbb{N}_{\geq2}$ be fixed, then:
\begin{itemize}
\item characteristic sequence of a finite set is $k$-automatic;
\item if sequence $\left(a\left(n\right)\right)_{n\in\mathbb{N}}$ differs from $\left(b\left(n\right)\right)_{n\in\mathbb{N}}$ only on finitely many terms and one of them is $k$-automatic so does the other one;
\item periodic sequence is $k$-automatic;
\item ultimately periodic sequence is $k$-automatic;
\end{itemize}
\end{lem}

Of course this does not exhaust all possible automatic sequences, but is enough to give some insight and be useful in our work. We should also notice what is the relation between automaticity in different bases.

\begin{lem} \label{automatonpower}
Sequence $\left(a\left(n\right)\right)_{n\in\mathbb{N}}$ is $k$-automatic if and only if it is $k^m$-automatic for all $m\in\mathbb{N}_{\geq 2}$.
\end{lem}
\begin{proof} (Sketch)
If we have $k$-automaton generating a sequence, then we can easily manipulate it to create $k^m$-automaton generating the same sequence, main idea is to take transition function to be $m$-th composition of the original transition function with itself (digit in base $k^m$ can be seen as $m$ digits in base $k$). \\
On the other hand, let $Q$ be set of states of the $k^m$-automaton generating a sequence, and $\rho$ be its transition function. We take $Q' = Q\times\{0,1,\ldots,k^{m-1}-1\}\times\{0,1,\ldots,m-1\}$ and
$$\rho'\left(\left(q,r,s\right),\sigma\right)=\begin{cases}
(q,kr+\sigma,s+1) & \text{if } s<m-1\\
(\rho(q,kr+\sigma),0,0) & \text{if } s=m-1\end{cases},$$
this way we accumulate base $k$ digits until we collect $m$ of them and then use the original transition function.
\end{proof}

\section{New results}
Lets start with some facts that we will be using in our proof
\begin{prop}(Legendre's formula \cite{legendre}) \label{legendre}
for any prime $p$ and positive integers $a,n$, we have
$$
v_{p^a}(n!) = \left\lfloor\frac{n-s_p\left(n\right)}{a\left(p-1\right)}\right\rfloor.
$$
\end{prop}
\begin{prop} \label{sum} (Result from \cite{stewart})
For any positive integers $b, c$ such that $\frac{\ln(b)}{\ln(c)}\not\in\mathbb{Q}$ there exists a constant $d$ such that for each integer $n>25$ there holds
$$
s_b\left(n\right) + s_{c}\left(n\right) > \frac{\log\log n}{\log \log \log n + d}-1.
$$
\end{prop}
Next proposition is known fact, but We haven't found it clearly stated anywhere, it can be easily proven using Dirichlet's approximation theorem or Equidistribution theorem.
\begin{prop} \label{word} For any positive integers $a, b, c$ such that $\frac{\ln(b)}{\ln(c)}\not\in\mathbb{Q}$ there exist infinitely many triples of non-negative integers $d,e,f$ with $1\leq f<b^e$ such that
$$
c^d = a\cdot b^e +f.
$$
In other words, there are infinitely many powers of $c$ with base $b$ notation starting with given string of digits. 
\end{prop}
After such introduction we can finally state our results. The following lemma and theorem are the main steps in proving when $\lns{b}$ is not automatic.

\begin{lem} \label{alwaysword}
Let $P$ be a non-empty finite set of prime numbers and $p$ be its biggest element. Let $a>0, k>1$ be integers. Then there exist an integer $a'$ such that $\max_{i\in P}\left\{s_i\left(a'\right)\right\}=s_p\left( a'\right)$ and $\left[a\right]_k$ is prefix of $\left[a'\right]_k$.
\end{lem}
\begin{proof}
If $k$ is not a power of $p$, then by Proposition \ref{word} there exist infinitely many triples $(d,e,f)$ of non-negative integers with $1\leq f+1<k^e$ such that
$$p^d = a\cdot k^e + (f+1).$$
Furthermore we have
$$s_p\left(p^d-1\right) = d\left(p-1\right)>\left(p-1\right)\frac{\ln\left(p^d-1\right)}{\ln(p)},$$
and from the definition of $s_q$, for any prime $q$ the following holds
$$s_q\left(p^d-1\right) < \left(q-1\right)\left(\frac{\ln\left(p^d-1\right)}{\ln(q)}+1\right).$$
Because $p$ is the biggest number in $P$, then for any $q\in P, q\neq p$, we have
$$s_q\left(p^d-1\right) - s_p\left(p^d-1\right)< \ln\left(p^d-1\right)\left(\frac{q-1}{\ln\left(q\right)}-\frac{p-1}{\ln\left(p\right)}\right) + q -1.$$
Right side of this inequality is negative for $d$ big enough, so because $0\leq f<k^e$ we can take $a' = p^d-1$. When $k=p^t$ we can notice that for any integer $d$
$$s_p\left( a\cdot p^{td}+p^{td}-1\right) =s_p\left(a\right)+ td\left(p-1\right)>\left(p-1\right)\left(\frac{\ln\left(a\cdot p^{td}+p^{td}-1\right)}{\ln(p)}-\frac{\ln\left(a\right)}{\ln(p)}-1\right),$$
and by similar argument it is enough to take $a'=a\cdot p^{td}+p^{td}-1$ for $d$ sufficiently large.
\end{proof}

\begin{thm} \label{sets}
Let $P$ be a finite set of prime numbers with at least two elements and $p$ be its biggest element, also let $c>0$ be a real number. Let us define sets
\begin{align*}
A_- &= \left\{n\in\mathbb{Z}_+:\;\max_{i\in P}\left\{s_i\left(n\right)\right\}=s_p\left(n\right)\right\},\\
A_+ &= \left\{n\in\mathbb{Z}_+:\;\max_{i\in P}\left\{s_i\left(n\right)\right\}-s_p\left(n\right)\geq c\right\}.
\end{align*}
Then there does not exist deterministic finite automaton with output that assigns one value to integers in $A_-$ and other value to those in $A_+$.
\end{thm}
\begin{proof}
Lets suppose that we have such an automaton $\left(Q, \Sigma_k, \rho, q_0, \Delta, \tau \right)$ for some $k$. Because $Q$ is finite, there exists some internal state $\mathcal{S}\in Q$ such that for infinitely many positive integers $c_1 < c_2 < \ldots$ we have $\rho\left(q_0,\left[p^{c_i}\right]_k\right) = \mathcal{S}$. Now, by Lemma \ref{alwaysword}, there exists an integer $a'\in A_-$ which can be obtained from $p^{c_1}$ by appending some suffix. Hence we can fix positive integers $e,f < k^e$ such that $a'=p^{c_1}\cdot k^e+f$. Let the sequence of digits $\left(f_1, f_2, \ldots, f_e\right)$ be a representation of $f$ in base $k$, possibly with added leading zeros. By $\mathcal{T}\in Q$ we denote an internal state such that $$\mathcal{T} = \rho\left(q_0, \left[a'\right]_k\right) = \rho\left(\mathcal{S},f_1f_2\ldots f_e\right).$$
This means that for every $i\in \mathbb N _+$ we have
$$\rho\left(q_0, \left[p^{c_i}\cdot k^e +f\right]_k\right) =\rho\left(q_0, \left[p^{c_i}\right]_kf_1f_2\ldots f_e\right) =\rho\left(\mathcal{S},f_1f_2\ldots f_e\right) = \mathcal{T},$$
and this implies that $\tau\left(\rho\left(q_0, \left[p^{c_i}\cdot k^e +f\right]_k\right)\right)=\tau\left(\mathcal{T}\right)$ does not depend on value of $i$.\\
On the other hand, when $c_i > \left\lceil \log_p(f)\right \rceil$ we have $s_p\left(p^{c_i}\cdot k^e +f\right)=s_p\left(k^e\right)+s_p\left(f\right)$ which is a constant. However, due to Proposition \ref{sum} we know that for any $q\in P, q \neq p$, the value of $s_q\left(p^{c_i}\cdot k^e+f\right)$ is increasing with $c_i$. Hence for $c_i$ big enough there holds $p^{c_i}\cdot k^e+f \in A_+$.\\
All but finitely many integers of the form $p^{c_i}\cdot k^e +f$ are elements of $A_+$ but at least one (namely $p^{c_1}\cdot k^e +f$) is an element of $A_-$. This proves that such automaton cannot assign different values to members of those two sets.
\end{proof}

Now we will show that $l_b\left(n!\right)$ can be automatic for some $b$.

\begin{lem} \label{bpa}
If $b=p^a, p\in \mathbb{P}, a\in \mathbb{N}$ then the sequence $\left(\ell_b\left(n!\right)\right)_{n\in\mathbb{N}}$ is $b$-automatic.
\end{lem}
\begin{proof}
First, we notice that $\ell_b\left(xy\right)=\ell_b\left(\ell_b\left(x\right)\ell_b\left(y\right)\right)$, so
$$\ell_b\left((bn)!\right) =\ell_b\left(\ell_b\left(n!\right)\prod_{i=n+1}^{bn} \ell_b\left(i\right) \right).$$
Because $\ell_b\left(bx\right) = \ell_b\left(x\right)$ we can rewrite the product in the following way
$$\ell_b\left((bn)!\right) =\ell_b\left(\ell_b\left(n!\right)\prod_{\substack{i=1 \\ b \nmid i}}^{bn} \ell_b\left(i\right) \right)=\ell_b\left(\ell_b\left(n!\right)\prod_{i=1}^{n} \ell_b\left(\prod_{j=1}^{j=b-1}j\right) \right).$$
We denote $m_i=\ell_b\left(i!\right)$ and obtain $\ell_b\left((bn)!\right) =\ell_b\left(\ell_b\left(n!\right) m_{b-1}^n \right)$. Now we take the string of digits $n_1n_2\ldots n_l = \left[n\right]_b$ and obtain the following formula
$$\ell_b\left(n!\right)=\ell_b\left(\left(n_1n_2\ldots n_l\right)!\right)=\ell_b\left(m_{n_l}\ell_b\left(\left(n_1n_2\ldots n_{l-1}\right)!\right)m_{b-1}^{\left(n_1n_2\ldots n_{l-1}\right)}\right),$$
which by iteration leads to
\begin{equation} \label{eq:factorial1}
\ell_b\left(\left(n_1n_2\ldots n_l\right)!\right)=\ell_b\left(m_{n_1}m_{n_2}\ldots m_{n_l} \ell_b\left(m_{b-1}^r\right)\right),
\end{equation}
Where $r=\left(n_1n_2\ldots n_{l-1}\right) +\ldots +\left(n_1n_2\right) + \left(n_1\right)$. Now, by Euler's Theorem $m_{b-1}^{\varphi(b)}\equiv m_{b-1}^{p^a-p^{a-1}}\equiv1 \pmod b$ so we only need to know the value of $r \pmod{p^a-p^{a-1}}$.
\begin{equation} \label{eq:factorial2}
r =\sum_{i=1}^{l-1}\left(b^{i-1}\sum_{j=1}^{l-i}n_j\right)
\equiv\sum_{i=1}^{l-1}n_i + p^{a-1}\sum_{i=1}^{l-2}\left(l-1-i\right)n_i \pmod{\left(p^a-p^{a-1}\right)}.
\end{equation}
Finally, we can define an automaton $\left(Q, \Sigma_b, \rho, q_0, \Delta, \tau \right)$ generating the sequence $\left(\ell_b\left(n!\right)\right)_{n\in\mathbb{N}}$ in the following way:
\begin{itemize}
\item the input alphabet $\Sigma_b=\{0,1,2,\ldots,b-1\}$;
\item the output alphabet $\Delta = \{1,2,\ldots,b-1\}$;
\item the set of states $Q=\Delta\times\Sigma_{p^a-p^{a-1}}\times\Sigma_{p-1}$;
\item the initial state $q_0=\left(1,0,0\right)$;
\item the output function $\tau\left(u,v,w\right)=\ell_b\left(u \cdot m_{b-1}^{v+p^{a-1}w}\right)$;
\item the transition function $$\rho\left(\left(u,v,w\right),s\right)=\left(\ell_b\left(u\cdot m_s\right),\left(v+s\right) \pmod{\left(p^a-p^{a-1}\right)},\left(w+v\right) \pmod{\left(p-1\right)}\right).$$
\end{itemize}
With such definition we have $\rho\left(q_0, \left[n\right]_b\right) = \left(u,v,w\right)$ where
\begin{itemize}
\item $u = \ell_b\left(m_{n_1}m_{n_2}\ldots m_{n_l}\right)$;\\
\item $v = \sum_{i=1}^{l-1}n_i \pmod {\left(p^a-p^{a-1}\right)}$; \\
\item $w = \sum_{i=1}^{l-2}\left(l-1-i\right)n_i \pmod{\left(p-1\right)}$.
\end{itemize}
Hence using equations \eqref{eq:factorial1} and \eqref{eq:factorial2} we see, that $\ell_b\left(n!\right) =\tau\left(u,v,w\right)$.
\end{proof}

Now we are ready to prove the following
\begin{thm}Let $b=p_1^{a_1}p_2^{a_2}\ldots$ with $a_1\left(p_1-1\right)\ge a_2\left(p_2-1\right)\ge \ldots$. The sequence $\left(\ell_b\left(n!\right)\right)_{n\in\mathbb{N}}$ is $p_1$-automatic if $a_1\left(p_1-1\right)> a_2\left(p_2-1\right)$ or $b=p_1^{a_1}$ and not automatic otherwise.
\end{thm}
\begin{proof}
Let $n \gg 0$. For $b={p_1}^{a_1}$ the sequence is $b$-automatic from Lemma \ref{bpa}, by Lemma \ref{automatonpower} it is also $p_1$-automatic. If $b$ has more than one prime factor and $a_1\left(p_1-1\right)> a_2\left(p_2-1\right)$ we take $b' = \frac{b}{p_1^{a_1}}$ so $p_1 \nmid b'$. From Proposition \ref{legendre} and the definition of $v_{b'}$ we have
$$v_{b'}\left(n!\right)=\min_{i>1}v_{p_i^{a_i}}\left(n!\right)=\min_{i>1}\left\lfloor\frac{n-s_{p_i}\left(n\right)}{a_i\left(p_i-1\right)}\right\rfloor > \left\lfloor\frac{n-s_{p_1}\left(n\right)}{a_1\left(p_1-1\right)}\right\rfloor=v_{p_1^{a_1}}\left(n!\right),$$
which leads to $b'|\ell_b\left(n!\right)$. Thus $\ell_b\left(n!\right)\in\left\{b',2b',3b',\ldots\left(p_1^{a_1}-1\right)b'\right\}$, so value of $\ell_b\left(n!\right)$ can be computed from value of $\ell_b\left(n!\right) \pmod{p_1^{a_1}}$. We also know, that there exist integers $c_1, c_2$ satisfying the equation
$$n! = b^{\left(v_{p_1^{a_1}}\left(n!\right)\right)} \ell_b\left(n!\right) +  b^{\left(v_{p_1^{a_1}}\left(n!\right)+1\right)} c_1 = p_1^{a_1\left(v_{p_1^{a_1}}\left(n!\right)\right)} \ell_{p_1^{a_1}}\left(n!\right) +  p_1^{a_1\left(v_{p_1^{a_1}}\left(n!\right)+1\right)} c_2.$$
After division of the above equality by $p_1^{a_1\left(v_{p_1^{a_1}}\left(n!\right)\right)}$ we obtain the following
$$\left(b'\right)^{\left(v_{p_1^{a_1}}\left(n!\right)\right)} \ell_b\left(n!\right) + p_1^{a_1} \left(b'\right)^{\left(v_{p_1^{a_1}}\left(n!\right)+1\right)} c_1 =  \ell_{p_1^{a_1}}\left(n!\right) +  p_1^{a_1} c_2.$$
Now, we can notice that $\ell_b\left(n!\right){\left(b'\right)}^{v_{p_1^{a_1}}\left(n!\right)} \equiv \ell_{p_1^{a_1}}\left(n!\right) \pmod{p_1^{a_1}}$, hence to finish this part of proof we just need to construct an $p_1$-automaton that returns the value of $v_{p_1^{a_1}}\left(n!\right) \pmod{\varphi(p_1^{a_1})}$. By Proposition \ref{legendre} this value can be computed from $\left(n-s_{p_1}(n)\right) \pmod{\varphi(p_1^{a_1})\cdot a_1(p_1-1)}$, and such expression is $p_1$-automatic as we already mentioned in Example \ref{simpleexa}.

Now, in the last case, when $a_1\left(p_1-1\right)=a_2\left(p_2-1\right)$, let $I = \{i : a_i\left(p_i-1\right) =a_1\left(p_1-1\right)\}$. Without loss of generality we can assume $p_1=\max_{i\in I} p_i$. By Legendre formula (Proposition \ref{legendre}) we have
\begin{align*}
\max_{i\in I}s_{p_i}\left(n\right)=s_{p_1}\left(n\right) &\Rightarrow v_{p_1^{a_1}}\left(n!\right) = \min_{i\in I} v_{p_i^{a_i}}\left(n!\right) \Rightarrow p_1^{a_1} \nmid \ell_b\left(n!\right),\\
\max_{i\in I}s_{p_i}\left(n\right)> a_1(p_1-1) + s_{p_1}\left(n\right) &\Rightarrow v_{p_1^{a_1}}\left(n!\right) > \min_{i\in I} v_{p_i^{a_i}}\left(n!\right) \Rightarrow p_1^{a_1} | \ell_b\left(n!\right).
\end{align*}
Hence, by Theorem \ref{sets}, there is no finite automaton that can, for given $n$, tell whether $p_1^{a_1}$ divides $\ell_b\left(n!\right)$ or not. This completes the proof, as finite automaton generating the sequence $\left(\ell_b\left(n!\right)\right)_{n\in\mathbb{N}}$ should distinguish those two sets.
\end{proof}

\end{document}